\documentclass[a4paper]{siamltex}
\usepackage{amssymb,amsmath}
\usepackage{url}
\usepackage{xspace}
\usepackage[colorlinks,bookmarks=true,bookmarksopen=true]{hyperref}

\usepackage[textsize=small]{todonotes}
\usepackage[ulem=normalem,draft]{changes}
\setdeletedmarkup{{\ifmmode\text{\sout{\ensuremath{#1}}}\else\sout{#1}\fi}}

\definechangesauthor{DW}

\newtheorem{remark}[theorem]{{\it Remark }\rm }

\newcommand{\dx}{\,\mathrm{d}x}

\newcommand{\bphi}{\bar\varphi}

\newcommand{\Pb}{\mbox{\rm (P)}\xspace}
\newcommand{\Pbred}{\mbox{\rm (P$_{\textrm{red}}$)}\xspace}
\newcommand{\Pbpc}{\mbox{\rm (P$_{\textrm{p.c.}}$)}\xspace}

\newcommand{\uad}{{U_{\rm ad}}}

\newcommand{\proj}{\operatorname{Proj}}
\newcommand{\sign}{\operatorname{sign}}

\def\R{\mathbb R}

\title{Analysis of Optimal Control Problems with an $L^0$ Term in the Cost Functional
\thanks{The first author was partially supported by Spanish Ministerio de Econom\'{\i}a y Competitividad under research project MTM2017-83185-P.
The second author was partially supported by the German Research Foundation (DFG) within
the priority program "Non-smooth and Complementarity-based Distributed Parameter Systems:
Simulation and Hierarchical Optimization" (SPP 1962) under grant number WA
3626/3-2.
}}

\author{Eduardo Casas\thanks{Departmento de Matem\'{a}tica Aplicada y Ciencias de la Computaci\'{o}n, E.T.S.I. Industriales y de Telecomunicaci\'on, Universidad de Cantabria, 39005 Santander, Spain, {\tt eduardo.casas@unican.es}.}
\and Daniel Wachsmuth\thanks{Institut f\"ur Mathematik, Universit\"at W\"urzburg, 97074 W\"urzburg, Germany, {\tt  daniel.wachsmuth@mathematik.uni-wuerzburg.de}.}
}

\begin{document}
\maketitle

\begin{abstract}
In this paper, we investigate optimal control problems subject to a semilinear elliptic partial differential equations.
The cost functional contains a term that measures the size of the support of the control, which is the so-called $L^0$-norm.
We provide necessary and sufficient optimality conditions of second-order. The sufficient second-order condition
is obtained by analyzing a partially convexified problem. Interestingly, the structure of the problem yields second-order conditions with
different bilinear forms for the necessary and for the sufficient condition.
\end{abstract}



\section{Introduction}
\label{S1}
In this paper, we study the following optimal control problem
\[
\Pb: \quad  \inf_{u \in \uad} J(u):=\int_\Omega L(x,y_u(x))\dx + \frac{\alpha}{2}\|u\|^2_{L^2(\Omega)} + \beta\|u\|_0
\]
where $y_u$ is the solution of the following semilinear elliptic equation
\begin{equation}
\left\{\begin{array}{rccl}Ay + a(x,y)& = & u & \mbox{in } \Omega,\\y & = & 0 & \mbox{on } \Gamma,\end{array}\right.
\label{E1.1}
\end{equation}

Here $A$ denotes an elliptic operator in the domain $\Omega \subset \mathbb{R}^n$, $1 \le n \le 3$, whose boundary is denoted by $\Gamma$,
and $a:\Omega \times \mathbb{R} \longrightarrow \mathbb{R}$ is a given function. Additionally, $L:\Omega \times \mathbb{R} \longrightarrow \mathbb{R}$ is another given function, $\alpha \ge 0$, $\beta > 0$, and
\[
\|u\|_0 = |\{x \in \Omega : u(x) \neq 0\}|,
\]
where $|B|$ denotes the Lebesgue measure of a set $B \subset \Omega$. Finally, we set
\[
\uad = \{u \in L^\infty(\Omega) : |u(x)| \le \gamma \text{ for a.a. } x \in \Omega\}
\]
$0 < \gamma \le \infty$. We assume that $\gamma < \infty$ if $\alpha = 0$. Precise assumptions on these data will be given in the next section.

We are interested in a second-order analysis of this problem. That is, we are looking for optimality conditions of second-order of necessary and sufficient type.
First-order necessary conditions are given by the famous Pontryagin maximum principle, see Theorem \ref{T3.6}.
Due to the properties of the cost functional, several difficulties will arise.
First, the functional $u\mapsto \frac{\alpha}{2}\|u\|^2_{L^2(\Omega)} + \beta\|u\|_0$ is non-smooth. We overcome this difficulty
by studying the convexification of this functional, which is continuously differentiable.
Still, this convexification is not twice differentiable which gives rise to the following observation:
two different bilinear forms connected to a second derivative of this functional are needed for
necessary and sufficient optimality conditions.
Second, the functional $u\mapsto \frac{\alpha}{2}\|u\|^2_{L^2(\Omega)} + \beta\|u\|_0$ is non-convex, and its convexification is not strictly convex.
Hence, we cannot expect that second-order derivatives of the Lagrangian associated with the control problem are coercive in $L^2(\Omega)$.
Here, we resort to techniques developed for bang-bang control problems, see, e.g. \cite{Casas2012}.

Optimal control problems with $L^0$-control cost were recently studied in \cite{ItoKunisch2014,Wachsmuth2019}.
The motivation is to obtain sparse controls, i.e., controls with small support.
In  the seminal paper \cite{Stadler2009}, this was addressed by using $\|u\|_{L^1(\Omega)}$
instead of $\|u\|_0$ in the cost functional.
Optimal control problems with  $L^0$-norms were also used to  enforce a particular control structure.
We refer to \cite{ClasonItoKunisch2016} for an application to switching control problems
and to \cite{ClasonKunisch2014} for control problems, where the control is allowed to take values only from a finite set.

The second-order analysis of the control problem with the convexified cost functional, mentioned above, is related to similar results for sparse control problems in \cite{CasasHerzogWachsmuth2012,CasasHerzogWachsmuth2017,HerzogStadlerWachsmuth2012},
In addition, we use recent results of \cite{CasasMateos2019} to reduce the cone of test directions in sufficient second-order conditions.

The main result of the paper is the derivation of sufficient second-order optimality conditions, given in Theorem \ref{T4.21}. It is proven by applying
similar results for the control problem with a convexified cost functional, which is studied in Section \ref{S4}.
As one might expect, the positivity requirements of sufficient conditions are stronger than those obtained from
necessary second-order conditions, the latter are studied in Section \ref{S3}.
The analysis relies on differentiability results related the control-to-state map associated
with the partial differential equation, these are provided in Section \ref{S2}.
Finally, let us mention that, under a certain assumption, a local (global) solution of the partially convexified problem is also a local (global) solution of \Pb; see Corollary \ref{C4.3}.

%
%
%
%
%

\section{Assumptions and preliminary results}
\label{S2}

Let us formulate the assumptions on our control problem \Pb.

(A1) We assume that $\Omega$ is an open and bounded domain in $\mathbb{R}^n$, $1 \le n \le 3$, with a Lipschitz boundary $\Gamma$, and $A$ denotes a second-order elliptic operator in $\Omega$ of the form
\[
Ay(x)=-\sum_{i,j=1}^{n}
\partial_{x_j}(a_{ij}(x)\partial_{x_i}y(x))
\]
with coefficients $a_{ij} \in L^\infty(\Omega)$ satisfying
\[
\Lambda_A|\xi|^2\leq\sum_{i,j=1}^n a_{ij}(x)\xi_i\xi_j\ \ \mbox{
}\forall \xi\in\mathbb{R}^n \ \mbox{ for a.e. } x\in\Omega
\]
for some $\Lambda_A > 0$.

(A2) $a : \Omega \times \mathbb{R} \longrightarrow \mathbb{R}$ is Carath\'eodory function of class $C^2$ w.r.t.~the second variable satisfying that $a(\cdot,0)\in L^{\bar p}(\Omega)$ with $\bar p > n/2$,
\[
\frac{\partial a}{\partial
y}(x,y) \geq 0 \quad  \mbox{ for a.e. } x \in \Omega,
\]
and for all $M>0$ there exists a constant $C_{a,M}>0$ such that
\[
\left|\frac{\partial a}{\partial y}(x,y)\right|+
\left|\frac{\partial^2 a}{\partial y^2}(x,y)\right| \leq C_{a,M}
\mbox{ for a.e. } x \in \Omega \mbox{ and } |y| \leq M.
\]
Furthermore, for every $M > 0$ and $\varepsilon > 0$ there exists $\delta > 0$, depending on $M$ and $\varepsilon$, such that
\[
\left|\frac{\partial^2a}{\partial y^2}(x,y_2) -
\frac{\partial^2a}{\partial y^2}(x,y_1)\right| < \varepsilon \
\mbox{ if } |y_1|, |y_2| \leq M,\ |y_2 - y_1| \le \delta, \mbox{ and for a.e. } x \in \Omega.
\]

(A3) $L : \Omega \times \mathbb{R} \longrightarrow \mathbb{R}$ is Carath\'eodory function of class $C^2$ w.r.t.~the second variable satisfying that $L(\cdot,0) \in L^1(\Omega)$, and for all $M>0$ there exist a constant $C_{L,M}>0$ and a function $\psi_M\in L^{\bar p}(\Omega)$ such that for every $|y| \le M$ and almost all $x \in \Omega$
\[
\left|\frac{\partial L}{\partial y}(x,y)\right|\leq \psi_M(x), \
\ \ \left|\frac{\partial^2L}{\partial y^2}(x,y)\right| \le C_{L,M}.
\]
In addition, for every $M > 0$ and $\varepsilon > 0$ there exists $\delta > 0$, depending on $M$ and $\varepsilon$, such that
\[
\left|\frac{\partial^2L}{\partial y^2}(x,y_2) -
\frac{\partial^2L}{\partial y^2}(x,y_1)\right| < \varepsilon \
\mbox{ if } |y_1|, |y_2| \leq M,\ |y_2 - y_1| \le \delta, \mbox{ and for a.e. } x \in
\Omega.
\]
In the case $\gamma = + \infty$, we also assume that there exists a function $\psi \in L^1(\Omega)$ such that $L(x,y) \ge \psi(x)$ for a.a.~$x \in \Omega$ and all $y \in \mathbb{R}$.

\subsection*{Discussion of the state equation}

As a consequence of Assumptions (A1) and (A2) we infer that for every $u \in L^p(\Omega)$ with $p > n/2$, the state equation \eqref{E1.1} has a unique solution $y_u \in H_0^1(\Omega) \cap C(\bar\Omega)$. The proof of this result is a quite standard combination of the Schauder's fix point theorem and the $L^\infty(\Omega)$ estimates \cite{Stampacchia65}. For the continuity of the solution in $\bar\Omega$ see, for instance, \cite[Theorem 8.30]{Gilbarg-Trudinger83}. Moreover, the mapping $S:L^p(\Omega) \longrightarrow H_0^1(\Omega) \cap C(\bar\Omega)$, defined by $S(u) = y_u$, is of class $C^2$. In the sequel, we will take $p = 2$ and we will denote by $z_v = S'(u)v$, which is the solution of
\begin{equation}
\left\{\begin{array}{rccc}\displaystyle Az + \frac{\partial a}{\partial y}(x,y)z& = & v & \mbox{in } \Omega\\z & = & 0 & \mbox{on } \Gamma\end{array}\right.
\label{E2.1}
\end{equation}

As usual, we consider the adjoint state equation associated with a control $u$
\begin{equation}
\left\{\begin{array}{rccc}\displaystyle A^*\varphi + \frac{\partial a}{\partial y}(x,y)\varphi& = & \displaystyle \frac{\partial L}{\partial y}(x,y) & \mbox{in } \Omega\\\varphi & = & 0 & \mbox{on } \Gamma\end{array}\right.
\label{E2.2}
\end{equation}
where $y = S(u)$ is the state corresponding to $u$. Because of the Assumptions (A3) on $L$, we have that $\varphi \in H_0^1(\Omega) \cap C(\bar\Omega)$. Moreover, for every $u \in L^2(\Omega)$ we have the estimates
\begin{equation}
\|y_u\|_{L^\infty(\Omega)} \le M_u = \bar C(\|a(\cdot,0)\|_{L^{\bar p}(\Omega)} + \|u\|_{L^2(\Omega)}),\ \|\varphi_u\|_{L^\infty(\Omega)} \le \bar C\|\psi_{M_u}\|_{L^{\bar p}(\Omega)}.
\label{E2.3}
\end{equation}
Consequently, if $\gamma<\infty$ there exists $M_\gamma >0 $ such that
\begin{equation}
\|y_u\|_\infty + \|\varphi_u\|_\infty \le M_\gamma \quad \forall u \in \uad. \label{E2.4}
\end{equation}

 Let us analyze the cost functional. First, we distinguish two parts in $J$. We set $J(u) = F(u) + \alpha\|u\|^2_{L^2(\Omega)} + \beta\|u\|_0$ with
\[
 F(u) = \int_\Omega L(x,y_u(x))\dx.
\]
Concerning the function $F:L^2(\Omega) \longrightarrow \mathbb{R}$, we have that it is of class $C^2$ and the first and second derivatives are given by
\begin{equation}
F'(u)v = \int_\Omega\varphi(x) v(x)\dx, \label{E2.5}
\end{equation}
and
\begin{equation}
F''(u)(v_1,v_2) = \int_\Omega\left(\frac{\partial^2L}{\partial y^2}(x,y(x)) - \varphi(x)\frac{\partial^2a}{\partial y^2}(x,y(x))\right)z_{v_1}(x)z_{v_2}(x)\dx,
\label{E2.6}
\end{equation}
where $y$ is the state associated with $u$, solution of \eqref{E1.1}, $\varphi$ is the adjoint state, solution of \eqref{E2.2}, and
$z_{v_i} = S'(u)v_i$ are the solution of \eqref{E2.1} for $v = v_i$, $i = 1, 2$. In the sequel we will use the identification $F'(u) = \varphi_u$ as an $L^2(\Omega)$ element.

\subsection*{Properties of $\|\cdot\|_0$ and existence of solutions}

For ease of presentation, let us define function $|\cdot|_0:\R\to\R$ by
\[
 |r|_0:= \begin{cases} 1 & \text{ if } r\ne 0,\\
          0 & \text{ if } r=0.
         \end{cases}
\]
Then, clearly $\|u\|_0 = \int_\Omega |u(x)|_0 \dx$ holds.
The function $|\cdot|_0:\R\to\R$ is discontinuous and lower semicontinuous, which implies
that $u \mapsto \|u\|_0$ is lower semicontinuous on $L^p(\Omega)$ with respect to the norm topology for all $p\in [1,\infty]$.
However, this mapping is not weakly lower semicontinuous on these $L^p(\Omega)$ spaces, see the example in \cite[Section 2.2]{Wachsmuth2019}.
In particular, the lack of weakly lower semicontinuity implies that the direct method of the calculus of variations cannot
be applied to prove existence of solutions.
Actually, the problem has no solution in general.
For an explicit example of such a situation see \cite[Section 4.5]{Wachsmuth2019}, where the state equation is a linear elliptic equation with Neumann boundary conditions.

This question will be addressed in Section \ref{S4}, where we provide
sufficient conditions for the existence of local solutions of \Pb.

\subsection*{Lipschitz estimates of $F'$ and $F''$ with respect to $z_v$}

In the subsequent second-order analysis, we will frequently
need the following technical results.

\begin{theorem}\label{T2.1}
Given $\bar u \in L^2(\Omega)$, there exists $\rho > 0$ and $C > 0$ such that
\[
\| F'(u) - F'(\bar u) \|_{L^2(\Omega)}\le C \|z_{u-\bar u}\|_{L^2(\Omega)}\quad \forall u \in \bar B_\rho(\bar u).
\]
\end{theorem}
\begin{proof}
Let us take $\rho \le 1$ to be fixed later. Let $u \in \bar B_\rho(\bar u) \subset \bar B_1(\bar u)$. Then, subtracting the equations satisfied by $y_u$ and $\bar y = y_{\bar u}$ we get with the mean value theorem
\[
A(y_u - \bar y) + \frac{\partial a}{\partial y}(x,y_\theta)(y_u - \bar y) = u - \bar u,
\]
where $y_\theta = \bar y + \theta(y_u - \bar y)$ for some measurable function $\theta:\Omega \longrightarrow [0,1]$. From the above equation we get
\begin{equation}
\|y_u - \bar y\|_{L^2(\Omega)} \le C_1\|u - \bar u\|_{L^2(\Omega)} \le C_1\rho.
\label{E2.7}
\end{equation}
On the other side, subtracting the equations satisfied by the adjoint states $\varphi_u$ and $\bar\varphi$ we get
\begin{equation}
A^*(\varphi_u - \bar\varphi) + \frac{\partial a}{\partial y}(x,\bar y)(\varphi_u - \bar\varphi) = \Big[\frac{\partial L}{\partial y}(x,y_u) - \frac{\partial L}{\partial y}(x,\bar y)\Big] + \big[\frac{\partial a}{\partial y}(x,\bar y) - \frac{\partial a}{\partial y}(x,y_u)\Big]\varphi_u.
\label{E2.8}
\end{equation}
Now, from \eqref{E2.3} we infer the existence of a constant $M$ such that
\[
\|y_u\|_{L^\infty(\Omega)} + \|\varphi_u\|_{L^\infty(\Omega)} \le M\quad \forall u \in \bar B_1(\bar u).
\]
Hence, using the Assumption (A2) and (A3) and the mean value theorem, we deduce from \eqref{E2.8} the existence of a constant $C_2$ such that
\begin{equation}
\| F'(u) - F'(\bar u) \|_{L^2(\Omega)} = \|\varphi_u - \bar\varphi\|_{L^2(\Omega)} \le C_2\|y_u - \bar y\|_{L^2(\Omega)}\quad \forall u \in \bar B_1(\bar u).
\label{E2.9}
\end{equation}
Arguing as in \cite[Corollary 2.8, (2.27)]{Casas2012}, there is $\rho>0$ such that
$
\|y_u-\bar y\|_{L^2(\Omega)} \le 2\|z_{u-\bar u}\|_{L^2(\Omega)}
$
for all $u\in\bar  B_\rho(\bar u)$.
Combining this inequality with \eqref{E2.9} the statement of the theorem follows.
\end{proof}

\begin{theorem}\label{T2.2}
For all $\epsilon>0$ there is $\rho>0$ such that
\[
 \left| (F''(u) -F''(\bar u))(u-\bar u)^2 \right| \le \epsilon \|z_{u-\bar u}\|_{L^2(\Omega)}^2
\]
for all $u\in B_\rho(\bar u)$.
\end{theorem}
\begin{proof}
It is a consequence of \cite[Lemma 2.7]{Casas2012} by selecting there $\uad = \bar B_1(\bar u)$.
\end{proof}

\begin{theorem}[{\cite[Lemma 2.6, (2.16)]{Casas2012}}]
\label{T2.3} Given $\bar u \in \uad$, there exists a constant $C_z$ such that
 \[
  \|z_v\|_{L^2(\Omega)} \le C_z \|v\|_{L^1(\Omega)}\quad \forall v \in L^1(\Omega).
 \]
\end{theorem}


\section{Necessary optimality conditions}
\label{S3}

Let $\bar u$ be a local minimum of \Pb in the sense of $L^2(\Omega)$. Let us define $\bar y:=y_{\bar u}$.
We will investigate necessary optimality conditions. The first step is the well-known
maximum principle, which can be considered a first-order necessary optimality condition.

\subsection{Pontryagin's maximum principle}

Let us define the Hamiltonian $H:\Omega \times \R^3\to\R$  associated with \Pb by
\[
 H(x,y,u,\varphi):= L(x,y) + \varphi u +  \frac\alpha2 u^2 + \beta|u|_0.
\]

\begin{theorem}[{\cite[Theorem 2]{Casas1994}}]
\label{T3.1}
 Let $\bar u$  be locally optimal for \Pb. Then there exists a uniquely determined
 adjoint state $\bar\varphi:=\varphi_{\bar u}$ solving the adjoint equation \eqref{E2.2} such that
 for almost all $x\in \Omega$
\begin{equation}\label{E3.1}
  H(x, \bar y(x), \bar u(x) , \bar\varphi(x)) \le
  H(x, \bar y(x), u , \bar\varphi(x))
  \quad \forall u\in [-\gamma,+\gamma]
\end{equation}
 is satisfied.
\end{theorem}

Let us note that the maximum principle implies a certain sparsity structure of the optimal controls.
To this end, let us first study a scalar optimization problem.

\begin{lemma}\label{L3.2}
Let $\varphi\in \R$ be given. Let $u^*$ be a global minimum of
 \[
  \min_{|u| \le \gamma} \varphi \cdot u +  \frac\alpha2 u^2 + \beta|u|_0.
 \]
 If $\alpha =0$ then
 \[
  (u^*,\varphi) \in  \{-\gamma\} \times \left[+\frac{\beta}{\gamma},+\infty\right)
  \ \cup\ \{0\} \times \left[-\frac{\beta}{\gamma},+\frac{\beta}{\gamma}\right]
  \ \cup\ \{\gamma\} \times \left(-\infty,-\frac{\beta}{\gamma}\right].
\]
If $\alpha>0$ then
one of the conditions of the following conditions is satisfied:
\begin{enumerate}
 \item\label{lem32_1} If $|\varphi| > \frac{\alpha\gamma}{2} + \frac{\beta}{\gamma}$ and $|\varphi| \ge \alpha\gamma \Rightarrow u^* = -\text{sign}(\varphi)\gamma$,
 \item\label{lem32_2} If $|\varphi| = \frac{\alpha\gamma}{2} + \frac{\beta}{\gamma}$ and $|\varphi| \ge \alpha\gamma \Rightarrow u^* = -\text{sign}(\varphi)\gamma$ {or} $u^* = 0$,
 \item\label{lem32_3} If $|\varphi| < \frac{\alpha\gamma}{2} + \frac{\beta}{\gamma}$ and $|\varphi| \ge \alpha\gamma \Rightarrow u^* = 0$,
 \item\label{lem32_4} If $\sqrt{2\alpha\beta} < |\varphi| < \alpha\gamma \Rightarrow u^* = -\frac{\varphi}{\alpha}$,
\item\label{lem32_5} If $\sqrt{2\alpha\beta} = |\varphi| < \alpha\gamma \Rightarrow u^* = -\frac{\varphi}{\alpha}$ {or} $u^* = 0$,
 \item\label{lem32_6} If $|\varphi| < \sqrt{2\alpha\beta} \Rightarrow u^* = 0$.
 \end{enumerate}
In particular, $u^*=0 $ if $|\varphi|<\sqrt{2\alpha\beta}$ for $\alpha>0$ or $|\varphi|< \frac\beta\gamma $ for $\alpha=0$.
%
%
\end{lemma}
\begin{proof}
Let $\alpha=0$. Then only the points $\{-\gamma,0,+\gamma\}$ are candidates for solutions of $\displaystyle\min_{|u| \le \gamma} \varphi \cdot u + \beta|u|_0$.
The claim follows by elementary computations.
The case $\alpha>0$ can be deduced from \cite[Lemma 3.5]{Wachsmuth2019}.
%
%
To proof the last statement observe that
$\frac{\alpha\gamma}2+\frac\beta\gamma - \sqrt{2\alpha\beta} = \frac12(\sqrt{\alpha\gamma} - \sqrt{\frac{2\beta}\gamma})^2 \ge0$.
Hence, if $|\varphi|<\sqrt{2\alpha\beta}$ then only case \ref{lem32_4} applies, and $u^*=0$ follows.
\end{proof}

\begin{corollary} \label{C3.3}
 Let $\bar u$ be a local minimum of \Pb with associated adjoint state $\bar\varphi$.
 Then we have for almost all $x\in \Omega$
 \begin{enumerate}
  \item for $\alpha>0$
  \begin{enumerate}
   \item   if $|\bar\varphi(x)|<\sqrt{2\alpha\beta}$ then $\bar u(x)=0$,
  \item if $\bar u(x)\ne0$ then $|\bar u(x)|\ge \min(\sqrt{\frac{2\beta}\alpha}, \gamma)$,

  \end{enumerate}
  \item for $\alpha=0$
  \begin{enumerate}
  \item if $|\bar\varphi(x)|<\frac\beta\gamma$ then $\bar u(x)=0$,
  \item if $\bar u(x)\ne0$ then $|\bar u(x)|=\gamma$.
 \end{enumerate}
 \end{enumerate}
\end{corollary}
\begin{proof}
 The claim is a direct consequence of 
 Theorem \ref{T3.1} and Lemma \ref{L3.2}.
\end{proof}

\begin{corollary}
Let $\gamma<+\infty$. Then there is $\beta^*\in(0,+\infty)$ such that $\bar u=0$ is the only stationary point of \Pb for every $\beta > \beta^ *$, and thus the only possible local (and global) solution of \Pb.
\end{corollary}
\begin{proof}
This is a consequence of \eqref{E2.4} and Corollary \ref{C3.3}. Actually, we can take
\[
\beta^* = \left\{\begin{array}{cl}\displaystyle \frac{M_\gamma^2}{2\alpha} & \text{if } \alpha > 0,\\\gamma M_\gamma  & \text{if } \alpha = 0\end{array}\right.
\]
with $M_\gamma$ as in \eqref{E2.4}.
\end{proof}

\begin{lemma}
 Suppose that $\gamma=+\infty$, $\alpha>0$. Then there is $\beta^*\in(0,+\infty)$ such that $\bar u=0$ is the only possible global solution of \Pb for every $\beta > \beta^ *$.
\end{lemma}
\begin{proof}
If $\bar u$ is a global solution of \Pb, then $J(\bar u) \le J(0)$. This implies with Assumption (A3)
$ 
\|\psi\|_{L^1(\Omega)} + \frac{\alpha}{2}\|\bar u\|^2_{L^2(\Omega)} \le J(\bar u) \le J(0),
$ 
therefore
\[
\|\bar u\|_{L^2(\Omega)} \le M_0 = \sqrt{\frac 2\alpha\Big(J(0) - \|\psi\|_{L^1(\Omega)}\Big)}.
\]
We infer from \eqref{E2.3} the inequalities
$ 
\|\bar y\|_{L^\infty(\Omega)} \le \bar M = \bar C(\|a(\cdot,0)\|_{L^{\bar p}(\Omega)} + M_0)
$ 
and $\|\bar\varphi\|_{L^\infty(\Omega)} \le \bar C\|\psi_{\bar M}\|_{L^{\bar p}(\Omega)}$. Then, it is enough to take
$ 
\beta^* = \frac{1}{2\alpha}\bar C^2\|\psi_{\bar M}\|^2_{L^{\bar p}(\Omega)}
$ 
to deduce from Corollary \ref{C3.3} that $\bar u = 0$ whenever $\beta > \beta^*$.
\end{proof}

Le us introduce the tangent cone of $\uad$ at $\bar u$,
which is given by
\[\begin{aligned}
 T_{\uad}(\bar u) = \{ v\in L^2(\Omega):
 &\ v(x) \ge 0\text{ if } \bar u(x) =-\gamma, \\
 &\ v(x) \le 0\text{ if } \bar u(x) =+\gamma \ \}.
\end{aligned}\]

\begin{lemma}\label{L3.5}
 Let $\bar u$ be satisfy the maximum principle \eqref{E3.1}. Then it holds
 \begin{equation}\label{E3.2}
  \int_\Omega (\bar \varphi + \alpha\bar u) v\dx \ge 0 \quad \forall v\in T_{\uad}(u): \ v(x) = 0 \text{ if } \bar u(x)=0.
 \end{equation}
 If $\alpha=0$ then for all such $v$
 \begin{equation}\label{E3.3}
  \int_\Omega \bar \varphi  v\dx \ge\frac\beta\gamma \|v\|_{L^1(\Omega)}
  \end{equation}
 holds.
\end{lemma}
\begin{proof}
First, suppose $\alpha>0$.
 Let us discuss the sign of the integrand in the claim pointwise. It suffices to investigate only points $x\in \Omega$ such that
 $\bar u(x)\ne0$. If $|\bar u(x)|<\gamma$ then $\alpha \bar u(x) + \bar\varphi(x)=0$ by the maximum principle and properties \ref{lem32_3} and \ref{lem32_4} of Lemma \ref{L3.2}.
 If $\bar u(x) = \gamma$, then $v(x)\le0$, and $\bar\varphi(x) \le -\alpha\gamma$ by {properties \ref{lem32_1} and \ref{lem32_2}}  of Lemma \ref{L3.2}.
 Hence, $(\bar \varphi(x) + \alpha\bar u(x)) v(x) \ge 0$ holds. Analogously, we argue for the case $\bar u(x) = -\gamma$.
Second, let $\alpha=0$. Then from the characterization in Lemma \ref{L3.2}, we have $|\varphi(x)|\ge \frac\beta\gamma$ for almost all
$x\in \Omega$ such that $\bar u(x)\ne0$. Moreover, $v(x)$ and $\varphi(x)$ have the same sign, hence the claim follows.
\end{proof}

\subsection{Second-order necessary optimality conditions}

In addition, we will prove second-order necessary conditions for \Pb.

\begin{theorem}\label{T3.6}
 Let $\bar u$ be locally optimal for \Pb. Then it holds
 \[
  F''(\bar u)(v,v) + \alpha \|v\|_{L^2(\Omega)}^2 \ge 0
 \]
for all $v \in C_{\bar u}$, where the critical cone $C_{\bar u}$ is given by
\[
 C_{\bar u}  = \{ v\in T_{\uad}(\bar u): \ v(x)=0 \text{ if } \bar u(x)=0 \text{ or } \bar \varphi(x) + \alpha\bar u(x)\ne0\}.
\]
\end{theorem}
Observe that {Lemma \ref{L3.2} leads to} $C_{\bar u} =\{0\}$ in the case $\alpha=0$.
\begin{proof}
Let $v\in C_{\bar u}$ be given. For $k\in \mathbb N$ define
\[
 v_k(x):=\begin{cases}
 0 & \text{ if } \gamma -\frac1k< |\bar u(x)| < \gamma,\\
          \proj_{[-k,k]}(v(x)) & \text{ otherwise}.
         \end{cases}
\]
Then $v_k$ is a feasible direction at $\bar u$,
and it holds
$J(\bar u+ tv_k) \ge J(\bar u)$ for all $t>0$ sufficiently small.
In addition, $\|\bar u+tv_k\|_0 \le \|\bar u\|_0$ by definition of $v_k$ and $v$.
Using this fact and expanding the differentiable parts, we find
\[
 \begin{aligned}
  0 &\le J(\bar u+ tv_k) - J(\bar u)\\
  &\le F(\bar u+tv_k) - F(\bar u) + \frac\alpha2 \|\bar u+tv_k\|^2_{L^2(\Omega)} - \frac\alpha2 \|\bar u\|^2_{L^2(\Omega)}\\
  & = t \int_\Omega (\bar \varphi + \alpha\bar u) v_k\dx  + \frac{t^2}2 F''(\bar u +\theta_t tv_k)v_k^2 + \frac\alpha2 t^2 \|v_k\|_{L^2(\Omega)}^2
 \end{aligned}
\]
with some $\theta_t\in(0,1)$. By construction of $v_k$, we have $\int_\Omega (\bar \varphi + \alpha\bar u) v_k\dx=0$. Dividing the inequality by $t^2$ and
passing to the limit $t\searrow0$, it follows
\[
   F''(\bar u)(v_k,v_k) + \alpha \|v_k\|_{L^2(\Omega)}^2 \ge 0
\]
for all $k$. Since $v_k \to v$ in $L^2(\Omega)$ for $k\to\infty$, the claim is proven.
\end{proof}

\subsection{Study of a reduced problem}
If $\bar u$ is a local solution of \Pb,
then it is also a local solution of
\[
\Pbred: \quad  \min_{u \in \uad(\bar u)} \int_\Omega L(x,y_u(x))\dx + \frac{\alpha}{2}\|u\|^2_{L^2(\Omega)}
\]
subject to the state equation \eqref{E1.1}, where the set $\uad(\bar u)$ is given by
\[
 \uad(\bar u) := \{u\in \uad: \ u(x) = 0 \text{ if } \bar u(x) =0\}.
\]
Due to the absence of the $L^0$-term, the problem \Pbred is a smooth optimal control problem.
Its first and second-order necessary optimality conditions are identical to
Lemma \ref{L3.5}, \eqref{E3.2}, and Theorem \ref{T3.6} above.

The first-order inequality \eqref{E3.3} of Lemma \ref{L3.5} in the case $\alpha=0$ is in fact a first-order
sufficient condition for local optimality in \Pbred.

\begin{corollary}
 Let $\alpha=0$.
 Let $\bar u$ satisfy the maximum principle for \Pb.
 Then there exists $\rho>0$ such that
 \[
  F(\bar u) + \frac\beta{2\gamma} \|u-\bar u\|_{L^1(\Omega)} \le F(u)
 \]
 for all $u\in B_\rho(\bar u) \cap \uad(\bar u)$, where $B_\rho(\bar u)$ denotes the $L^2(\Omega)$-ball around $\bar u$.
\end{corollary}
\begin{proof}
 First, we have the expansion
 \[
  F(u) - F(\bar u) = F'(\bar u)(u-\bar u) + (F'(\bar u + \theta(u-\bar u))-F'(\bar u))(u-\bar u)
 \]
 with some $\theta\in (0,1)$.
 Using the properties of $F$,
 there is $\rho>0$ such that $ \|F'(\bar u + \theta(u-\bar u))-F'(\bar u)\|_{L^\infty(\Omega)} \le \frac\beta{2\gamma} $ for all $u\in B_\rho(\bar u)$,
 see \cite[Lemma 2.5]{Casas2012}.
 The claim follows from \eqref{E2.5} and Lemma \ref{L3.5}, \eqref{E3.3}.
\end{proof}


Similarly, we can formulate a second-order sufficient condition for \Pbred in the case $\alpha>0$.

\begin{corollary}
 Let $\alpha>0$.
 Let $\bar u$ satisfy the maximum principle for \Pb.
 Assume that
 \[
  F''(\bar u)(v,v) + \alpha \|v\|_{L^2(\Omega)}^2 > 0 \quad \forall v\in C_{\bar u} \setminus\{0\},
 \]
 where $C_{\bar u}$ is as in Theorem \ref{T3.6}. Then $\bar u$ is locally optimal for \Pbred in the $L^2(\Omega)$-sense.
\end{corollary}
\begin{proof}
 This is \cite[Theorem 2.2]{Casas2012} applied to \Pbred.
\end{proof}

\section{Second-order sufficient optimality conditions}
\label{S4}

In this section, we will study sufficient optimality conditions of second order.
First, we will develop such a condition for a partially convexified problem,
where
the term
\[
 j(u) := \frac\alpha2 \|u\|_{L^2(\Omega)}^2 + \beta \|u\|_0
\]
is replaced by its convexification on the feasible intervall $[-\gamma,\gamma]$.

\subsection{Partially convexified problem}

The convexification of $j$ will be denoted by
\[
 G(u):=\int_\Omega g(u(x))\dx,
\]
where $g:\R\to\R$ is the convexification of the integrand of $j$.
Here, we have to distinguish two cases.
In case $\sqrt{\frac{2\beta}\alpha} < \gamma$, the function
$g$ is given by
\[
 g(u) = \begin{cases}
         \frac\alpha2 u^2 + \beta & \text{ if } |u| \ge \sqrt{\frac{2\beta}\alpha},\\
         \sqrt{2\alpha\beta} |u| & \text{ if } |u| < \sqrt{\frac{2\beta}\alpha},\\
        \end{cases}
\]
with its directional derivative at $u\in \R$ in direction $v$ given by
\begin{equation}\label{E4.1}
 g'(u;v) =\begin{cases}
         \alpha u v  & \text{ if } |u| \ge \sqrt{\frac{2\beta}\alpha},\\
         \sqrt{2\alpha\beta} \sign(u)v & \text{ if } 0<|u| < \sqrt{\frac{2\beta}\alpha},\\
         \sqrt{2\alpha\beta} |v| & \text{ if } u=0.
        \end{cases}
\end{equation}
In addition, we have the important equality
\begin{equation}\label{eq405}
 g(u) = \frac\alpha2 u^2 + \beta |u|_0 \ \Leftrightarrow \ u=0 \text{ or } |u|\ge \sqrt{\frac{2\beta}\alpha}.
\end{equation}
In the case $\sqrt{\frac{2\beta}\alpha} \ge \gamma$ with $\alpha \ge 0$ the integrand $g$ of the convex hull of $j$ is given
by
\[
 g(u) = \left( \frac{\alpha\gamma}2 + \frac\beta\gamma \right) |u|.
\]
Please compare also with the distinction of cases in Lemma \ref{L3.2} and Corollary \ref{C3.3}.
The function $g$ is continuously differentiable on $\R\setminus\{0\}$.
In addition, $G$ is weakly lower semicontinuous on $L^2(\Omega)$.

The partially convexified problem is defined as
\[
 \Pbpc: \ \min_{u\in \uad} F(u) + G(u).
\]
The objective functional is the sum of a smooth function $F$ and a convex function $G$.
This functional is weakly lower semicontinuous, hence \Pbpc is solvable.
Its first-order optimality conditions are as follows.

\begin{theorem}[First-order necessary conditions for \Pbpc]
\label{T4.1}
Let $\bar u\in \uad$ be locally optimal for \Pbpc.
Let $\bphi$ denote the associated adjoint state.
Then the variational inequality
\begin{equation}\label{eq41}
 \int_\Omega  \bphi(x)v(x) + g'(\bar u(x);v(x))\dx \ge 0
\end{equation}
is satisfied for all $v\in T_\uad(\bar u)$.
\end{theorem}

By standard arguments, inequality \eqref{eq41} is equivalent to the pointwise inequality
\begin{equation}\label{eq42}
 \bphi(x)v + g'(\bar u(x);v) \ge 0
 \quad \text{ for a.a.}\, x\in \Omega, \ \forall v\in T_{[-\gamma,\gamma]}(\bar u(x)),
\end{equation}
where
\[
T_{[-\gamma,\gamma]}(\bar u(x)) = \{v \in \mathbb{R} : v \left\{\begin{array}{cl}\ge 0 & \text{if } \bar u(x) = -\gamma,\\\le 0 & \text{if } \bar u(x) = +\gamma.\end{array}\right.\}
\]
In addition, this inequality is equivalent to the Pontryagin maximum principle for \Pbpc due to the
convexity of $g$.
In the case $\sqrt{\frac{2\beta}\alpha} \ge \gamma$, the function $G$ is
a multiple of the $L^1(\Omega)$-norm, which implies a certain sparsity structure of optimal
controls. A similar result is true for the $\sqrt{\frac{2\beta}\alpha} < \gamma$ as well.
Here, we have the following result, which is an analogue to Corollary \ref{C3.3}.

\begin{lemma}
\label{L4.2}
Suppose $\sqrt{\frac{2\beta}\alpha} < \gamma$.
Let $\bar u$ be stationary point of \Pbpc.
Then we have the implications
\[\begin{aligned}
  |\bphi(x)| < \sqrt{2\alpha\beta}  \ &\Rightarrow \ \bar u(x)=0,\\
  \bphi(x) = -\sqrt{2\alpha\beta}  \ &\Rightarrow \ \bar u(x) \in [0,+\sqrt{\frac{2\beta}{\alpha}}],\\
    \bphi(x) = +\sqrt{2\alpha\beta}  \ &\Rightarrow \ \bar u(x) \in [-\sqrt{\frac{2\beta}{\alpha}},0],\\
  |\bphi(x)| > \sqrt{2\alpha\beta}  \ &\Rightarrow \ \bar u(x) = \proj_{[-\gamma,\gamma]}(-\frac1\alpha \bphi(x)),
  \end{aligned}
\]
 for almost all $x\in \Omega$.
If
 \[
  \big|\{x\in \Omega: \  |\bphi(x)| = \sqrt{2\alpha\beta} \}\big| =0
 \]
 is satisfied
then
\[
 \bar u(x) \ne0 \ \Rightarrow \ |\bar u(x)| \ge \sqrt{\frac{2\beta}{\alpha}}
\]
holds for almost all $x\in \Omega$.
\end{lemma}
\begin{proof}
Let $|\bphi(x)| < \sqrt{2\alpha\beta}$.
Suppose $\bar u(x)>0$. Choose $v<0$. Then, \eqref{eq42} implies that {$\bphi(x)v + g'(\bar u(x),v)  \ge 0$}. From the expression for $g'(\bar u(x),v)$ we get
\begin{align*}
& \bphi(x) + \alpha \bar u(x) \le 0 \text{ if } \bar u(x) \ge \sqrt{\frac{2\beta}{\alpha}} \Rightarrow \sqrt{2\alpha\beta} \le \alpha\bar u(x) \le -\bphi(x),\\
& \bphi(x) + \sqrt{2\alpha\beta} \le 0 \text{ if } \bar u(x) < \sqrt{\frac{2\beta}{\alpha}} \Rightarrow \sqrt{2\alpha\beta} \le -\bphi(x).
\end{align*}
In any of these case we get a contradiction with the fact that $|\bphi(x)| < \sqrt{2\alpha\beta}$. A similar contradiction is obtained for the case $\bar u(x) < 0$.
Now, we assume that $\bphi(x) = -\sqrt{2\alpha\beta}$ and we prove that $0 \le \bar u(x) \le \sqrt{\frac{2\beta}{\alpha}}$. We argue by contradiction and we assume that $\bar u(x) > \sqrt{\frac{2\beta}{\alpha}}$. Taking again $v < 0$ in \eqref{eq42} we deduce that
\[
0 \ge \bphi(x) + \alpha\bar u(x) > -\sqrt{2\alpha\beta} + \alpha\sqrt{\frac{2\beta}{\alpha}}  = 0
\]
and we get a contradiction. If $\bar u(x) < 0$, then selecting $v > 0$ it is easy to check that $g'(\bar u(x);v) < 0$ and, hence, $\bphi(x)v + g'(\bar u(x);v) < 0$, which contradicts \eqref{eq42}. Analogously we prove the case $\bphi(x) = +\sqrt{2\alpha\beta}$.

Finally, we analyze the case $|\bphi(x)| > \sqrt{2\alpha\beta}$. First we prove that $|\bar u(x)| > \sqrt{\frac{2\beta}{\alpha}}$. Indeed, in the contrary case \eqref{eq42} implies that
\begin{align*}
&(\bphi(x) + \sqrt{2\alpha\beta}\sign(\bar u(x)))v \ge 0 \text{ if } \bar u(x) \neq 0,\\
&\bphi(x)v + \sqrt{2\alpha\beta}|v| \ge 0 \text{ if } \bar u(x) = 0
\end{align*}
holds for every $v \in \mathbb{R}$. However, taking $v = -\sign(\bphi(x))$ we get a contradiction. Hence, we have that $|\bar u(x)| > \sqrt{\frac{2\alpha}{\beta}}$ and, consequently, \eqref{eq42} implies that $(\bphi(x) + \alpha\bar u(x))v \ge 0$ $\forall v \in T_{[-\gamma,\gamma]}(\bar u(x))$. Taking into account that $v - \bar u(x) \in T_{[-\gamma,\gamma]}(\bar u(x))$ for every $v \in [-\gamma,+\gamma]$, we have that $(\bphi(x) + \alpha\bar u(x))(v-\bar u(x)) \ge 0$ $\forall v \in [-\gamma,+\gamma]$, which is well known to be equivalent to $\bar u(x) = \proj_{[-\gamma,\gamma]}(-\frac1\alpha \bphi(x))$.

Finally, if $|\bphi(x)|= \sqrt{2\alpha\beta}$ is only true on a set of zero measure, then $0<|\bar u(x)| < \sqrt{\frac{2\beta}{\alpha}}$ is only true on a set of zero measure, which proves the second claim.
\end{proof}

\begin{corollary}
Suppose $\sqrt{\frac{2\beta}\alpha} < \gamma$.
 Let $\bar u$ be a local (global) solution of \Pbpc. Assume that
 \[
  \big|\{x\in \Omega: \  |\bphi(x)| = \sqrt{2\alpha\beta} \}\big| =0
 \]
 Then $\bar u$ is a local (global) solution of \Pb.
 \label{C4.3}
\end{corollary}
\begin{proof}
Let $\bar u$ be a minimum of $F + G$  on some neighborhood $U$ of $\bar u$, i.e., $F(\bar u) + G(\bar u) \le F(u) + G(u)$ for all $u\in U\cap \uad$.
 Then the conclusion of Theorem \ref{T4.1} is valid, and property \eqref{eq42} is satisfied.
 Using the result of Lemma \ref{L4.2} it follows that for almost all $x\in \Omega$ we have $\bar u(x)=0$
 or $|\bar u(x)| \ge\sqrt{\frac{2\beta}{\alpha}}$. By \eqref{eq405}, this implies $j(\bar u) = G(\bar u)$.

 Let now $u\in U\cap \uad$ be given. Then we have the chain of inequalities
 \[
  F(\bar u) + j(\bar u) = F(\bar u) + G(\bar u) \le F(u) + G(u) \le F(u) + j(u),
 \]
which proves the claim.
\end{proof}

Let us conclude this section with a remark on second-order optimality conditions of \Pbpc in the
case $\sqrt{\frac{2\beta}\alpha} \ge \gamma$.
Here, the function $G$ is a multiple of the $L^1(\Omega)$-norm.
Such problems are well studied in the literature.
A sufficient optimality condition is given by \cite[Theorem 3.6]{Casas2012}.
Hence, we will consider the case $\sqrt{\frac{2\beta}\alpha} < \gamma$ from now on.

\subsection{Second-order optimality conditions for the partially convexified problem}

Let us assume now  $\sqrt{\frac{2\beta}\alpha} < \gamma$. In this case, the directional derivative of $g$ at $u$ in direction $h$ is given by
\[
 g'(u;h) =\begin{cases}
         \alpha u h  & \text{ if } |u| \ge \sqrt{\frac{2\beta}\alpha},\\
         \sqrt{2\alpha\beta} \sign(u)h & \text{ if } 0<|u| < \sqrt{\frac{2\beta}\alpha},\\
         \sqrt{2\alpha\beta} |h| & \text{ if } u=0.
        \end{cases}
\]
Clearly $g$ is not differentiable at $u=0$, and it is not twice differentiable at $\pm \sqrt{\frac{2\beta}\alpha}$.
Still let us introduce some kind of second-order directional derivative defined by
\[
 g''(u;h^2) := \begin{cases}
 \alpha h^2  & \text{ if } |u| > \sqrt{\frac{2\beta}\alpha},\\
 \alpha h^2  & \text{ if } u = \sqrt{\frac{2\beta}\alpha}, \ h \ge 0,\\
 \alpha h^2  & \text{ if } u = -\sqrt{\frac{2\beta}\alpha}, \ h \le 0,\\
0 & \text{ otherwise,}
        \end{cases}
\]
and
\[
 G''(u,h^2) := \int_\Omega g''(u(x); h(x)^2 ) \dx.
\]
This choice of $g''$ is justified by the Taylor expansion provided by the next lemma.

\begin{lemma}
Let $u\in \R$ be given.
Then there is $\delta=\delta(u)>0$ such that
\[
 g(u+h) - g(u) - g'(u;h) - \frac12 g''(u;h^2) = 0
\]
for all $h$ with $|h|\le \delta$.
\end{lemma}
\begin{proof}
Obviously, the claim is true if $u\not\in \{-\sqrt{\frac{2\beta}\alpha},0,\sqrt{\frac{2\beta}\alpha}\}$,
because $g$ is a polynomial of degree at most two near such values of $u$ with second-derivative given by the expression for $g''(u,\cdot)$ above.
First, let us consider the case $u=0$. Let $|h|<\sqrt{\frac{2\beta}\alpha}$. Then
clearly $g(u+h) - g(u) - g'(u;h) - \frac12 g''(u;h^2)=0$.
Second, let $u= \sqrt{\frac{2\beta}\alpha}$  and $h>0$. Then both $u$ and $u+h$ lie on the quadratic branch of $g$, which means
that the remainder is zero. If $-\sqrt{\frac{2\beta}\alpha} < h <0$,
 \[
 g(u+h) - g(u) - g'(u;h) - \frac12 g''(u;h^2) = \sqrt{2\alpha\beta}\big( (u+h) - u - h  \big)- 0 =0.
 \]
 The case $u=- \sqrt{\frac{2\beta}\alpha}$ follows analogously.
\end{proof}

For second-order optimality conditions, only lower bounds of this remainder term are of importance.
Here, we have the following result.

\begin{lemma}\label{L4.6}
Let $u,h\in  \R$ be given.
Then it holds
\[
  g(u+h) - g(u) - g'(u;h) - \frac12 g''(u;h^2) \ge
  \begin{cases} 0 & \text{ if } |u| \le \sqrt{\frac{2\beta}\alpha}\\
   -\frac\alpha2 \left[\Big(\sqrt{\frac{2\beta}\alpha}-|u+h|\Big)_+\right]^2 & \text{ if } |u| > \sqrt{\frac{2\beta}\alpha}.
  \end{cases}
  \]
\end{lemma}
\begin{proof}
 If $|u|<\sqrt{\frac{2\beta}\alpha}$ then the claim follows from the convexity of $g$.
 Let now $u=\sqrt{\frac{2\beta}\alpha}$. Then for $h<0$ the claim follows again from the convexity of $g$,
 while for $h>0$ the claim follows from the quadratic nature of $g$ on $[\sqrt{\frac{2\beta}\alpha}, +\infty)$.

 Consider now the case $|u|>\sqrt{\frac{2\beta}\alpha}$. If $|u+h|\ge\sqrt{\frac{2\beta}\alpha}$ then $g(u+h) - g(u) - g'(u;h) - \frac12 g''(u;h^2)=0$.
Suppose $|u+h|<\sqrt{\frac{2\beta}\alpha}$. Then we find
\[\begin{aligned}
 g(u+h) - g(u) - g'(u;h) - \frac12 g''(u;h^2) &=  \sqrt{2\alpha\beta}|u+h| - \frac\alpha2 (u+h)^2 - \beta\\
& =-\frac\alpha2 \left(\sqrt{\frac{2\beta}\alpha} - |u+h|\right)^2,
\end{aligned}\]
which finishes the proof.
 \end{proof}

Using this pointwise inequality, we can prove a lower bound of a Taylor expansion of the
integral functional $G$.

\begin{lemma}\label{L4.7}
Let $p>2$.
Let $u\in L^p(\Omega)$, then
 \[
  G(u+h) - G(u) - G'(u;h) - \frac12G''(u;h^2) \ge o(\|h\|_{L^p(\Omega_u)}^2)
 \]
for $h\to0$ in $L^p(\Omega)$ with
 \[
  \Omega_{u} := \left\{ x\in\Omega:\ |u(x)| > \sqrt{\frac{2\beta}\alpha}\right\}.
 \]
\end{lemma}
\begin{proof}
Let $h_k$ be given such that $h_k \to 0$ in $L^p(\Omega)$ and pointwise, $p>2$.
Due to Lemma \ref{L4.6}, we have the lower bound
\[
   G(u+h_k) - G(u) - G'(u;h_k) - \frac12G''(u;h_k^2)
   \ge
   -\frac\alpha2 \int_{\Omega_{u,u+h_k}}\left(\sqrt{\frac{2\beta}\alpha} - |u+h_k|\right)^2\dx,
\]
where
\[
 \Omega_{u,u+h_k}:=\left\{x\in\Omega:\ |u(x)|>\sqrt{\frac{2\beta}\alpha}>|u(x)+h_k(x)| \right\}.
\]
Clearly the measure of $\Omega_{u,u+h_k}$ tends to zero for $k\to\infty$.
In addition $(|u(x)+h_k(x)|-\sqrt{\frac{2\beta}\alpha}) ^2 \le h_k(x)^2$ holds on this set.
Using H\"older's inequality, we thus find
\[
   G(u+h_k) - G(u) - G'(u;h_k) - \frac12G''(u;h_k^2) \ge -\frac\alpha2 |\Omega_{u,u+h_k}|^{1-\frac2p} \|h_k\|_{L^p(\Omega)}^2,
\]
which proves the claim.
\end{proof}

\begin{remark}\label{R4.8}
Let us comment that the previous result is not true in general for $p=2$.
To this end, let $\Omega=(0,1)$, $\alpha=\beta=1$, $u(x) = 2>\sqrt{\frac{2\beta}\alpha}=\sqrt2$, $h_k(x):=-\chi_{(0,\frac1k)}(x)$.
Then $\|h_k\|_{L^2(\Omega)}^2 = \frac1k$. In addition, we have
\begin{multline*}
 G(u+h_k) - G(u) - G'(u;h_k) - \frac12 G''(u;h_k^2)
 = -\int_\Omega \frac12 \left[\Big(\sqrt2-|u(x)+h_k(x)|\Big)_+\right]^2\dx\\
 = - \frac12 \frac1k (\sqrt2-1)^2 = - \frac12(\sqrt2-1)^2 \|h_k\|_{L^2(\Omega)}^2.
\end{multline*}
\end{remark}

Let us introduce the critical cone for \Pbpc by
\[
C_{\textrm{pc,}\bar u} := \{ v\in T_\uad(\bar u): \ \bar\varphi(x)v(x) + g'(\bar u(x);v(x)) = 0 \}.
\]
We have the  following characterization of $ C_{\textrm{pc,}\bar u}$.

\begin{lemma}
Let $\bar u$ be a stationary point of \Pbpc.
Then, $v\in  C_{\textrm{pc,}\bar u}$ if and only if $v \in T_\uad(\bar u)$ and the following conditions hold for almost all $x\in \Omega$:
\begin{enumerate}
 \item If $|\bar\varphi(x)|< \sqrt{2\alpha\beta}$ then $v(x)=0$,
 \item If $\bar\varphi(x)= +\sqrt{2\alpha\beta}$ and $\bar u(x) = 0$ then $v(x)\le0$,
 \item If $\bar\varphi(x)= -\sqrt{2\alpha\beta}$ and $\bar u(x) = 0$ then $v(x)\ge0$,
 \item If $|\bar\varphi(x)|>\alpha\gamma$  then $v(x)=0$.
\end{enumerate}
\end{lemma}
\begin{proof}
This is a direct consequence of Lemma \ref{L4.2} and the form of $g'$, see \eqref{E4.1}.
%
\end{proof}

\begin{remark}
The conditions on $v(x)$ in case $\bar\varphi(x)= \sqrt{2\alpha\beta}$ also appear in the
critical cone associated to $L^1(\Omega)$-optimal control problems, see \cite[Proposition 3.3]{Casas2012}.
\end{remark}

\begin{remark}
 Let us compare the critical cones $C_{\textrm{pc,}\bar u}$ and $C_{\bar u}$, where the latter was defined in Theorem \ref{T3.6}.
 Clearly it holds  $C_{\bar u} \subset  C_{\textrm{pc,}\bar u}$ for any feasible control $\bar u$.
\end{remark}

\begin{theorem}
 Let $\bar u$ be locally optimal for \Pbpc.
 Then it holds
 \[
  F''(\bar u)v^2 + G''(\bar u; v^2 ) \ge0 \quad\forall v\in
 C_{\textrm{pc,}\bar u}.
 \]
\end{theorem}
\begin{proof}
Let $v\in C_{\textrm{pc,}\bar u}$ be given. For $k\in \mathbb N$ define
\[
 v_k(x):=\begin{cases}
 0 & \text{ if } \gamma - \frac1k < |\bar u(x)| < \gamma,\\
 0 & \text{ if } \sqrt{\frac{2\beta}\alpha} <  |\bar u(x)| < \sqrt{\frac{2\beta}\alpha} + \frac1k,\\
          \proj_{[-k,k]}(v(x)) & \text{ otherwise}.
         \end{cases}
\]
Then $v_k\in C_{\textrm{pc,}\bar u}\cap L^\infty(\Omega)$ is a feasible direction at $\bar u$, which implies
$F(\bar u + tv_k) + G(\bar u + tv_k) - F(\bar u) - G(\bar u)\ge0$ for all $t>0$ small enough.
In addition, Lemma \ref{L4.6} and the construction of $v_k$ implies that
$ 
 G(\bar u + tv_k) - G(\bar u) - t G'(\bar u; v_k) \ge \frac{t^2}2 G''(\bar u; v_k^2)
$ 
for all $t>0$ small enough. For such {a} small $t$ we have
\[
\begin{aligned}
 0 & \le F(\bar u + tv_k) + G(\bar u + tv_k) - F(\bar u) - G(\bar u) \\
 & \le t F'(\bar u )v_k + \frac{t^2}2 F''(\bar u + \theta_t t v_k)v_k^2 + t G'(\bar u; v_k) + \frac{t^2}2 G''(\bar u; v_k^2)\\
 & =  \frac{t^2}2F''(\bar u + \theta_t t v_k)v_k^2 + \frac{t^2}2 G''(\bar u; v_k^2)
\end{aligned}
\]
with some $\theta_t\in (0,1)$. Dividing by $t^2$ and passing to the limit yields $F''(\bar u)v_k^2 + G''(\bar u; v_k^2 ) \ge0$
for all $k$. Passing to the limit $k\to \infty$ proves the claim.
\end{proof}

For second-order sufficient optimality conditions, we will work with the following extensions of the critical cone
$C_{\textrm{pc,}\bar u}$.
Similarly to \cite{CasasMateos2019}, we define for $\tau>0$
\begin{equation}\label{eqdtau}
 \begin{aligned}
 D^\tau_{\bar u}:=\{ v\in T_\uad(\bar u):
  \ & v(x) \ge 0 \text{ if } \bar u(x)=0 \text{ and } \bar\varphi(x)= -\sqrt{2\alpha\beta},\\
  & v(x) \le 0 \text{ if }  \bar u(x)=0 \text{ and } \bar\varphi(x)= \sqrt{2\alpha\beta},\\
  & v(x) = 0 \text{ if } |\bphi(x)| \le \sqrt{2\alpha\beta} -\tau \text{ or }|\bphi(x)| \ge \alpha\gamma + \tau &\},
 \end{aligned}
\end{equation}
\begin{equation}\label{eqetau}
 E^\tau_{\bar u}:=\{ v\in T_\uad(\bar u): F'(\bar u)v + G'(\bar u;v) \le \tau \|z_v\|_{L^2(\Omega)}  \},
\end{equation}
and
\begin{equation}\label{eqctau}
 C^\tau_{\bar u} := D^\tau_{\bar u} \cap E^\tau_{\bar u}.
\end{equation}
Directions not contained in $D^\tau_{\bar u}$ give rise to positive lower bounds from first-order
derivatives. Precisely, we have the following
\begin{lemma}\label{L4.13}
 Let $\bar u$ satisfy the necessary optimality conditions of \Pbpc and assume that $\tau < \sqrt{2\alpha\beta}$.
 Let $w\in T_{\uad}(\bar u)$.
 Define the set
 \[\begin{aligned}
  \Omega_{\bar u,w}:=\{ x\in \Omega:
  \ & w(x) < 0 \text{ if } \bar u(x)=0 \text{ and } \bar\varphi(x)= -\sqrt{2\alpha\beta},\\
  & w(x)  > 0 \text{ if }  \bar u(x)=0 \text{ and } \bar\varphi(x)= \sqrt{2\alpha\beta},\\
  & w(x) \ne 0 \text{ if } |\bphi(x)| \le \sqrt{2\alpha\beta} -\tau \text{ or }|\bphi(x)| \ge \alpha\gamma + \tau \}.
 \end{aligned}
 \]
Then we have
\[
 F'(\bar u)w + G'(\bar u;w) \ge \tau\|w\|_{L^1(\Omega_{\bar u,w})}.
\]
\end{lemma}
\begin{proof}
Take $x\in \Omega$ such that $w(x)  < 0$, $\bar u(x)=0$, and $\bar\varphi(x)= -\sqrt{2\alpha\beta}$.
Then it holds $\bphi(x) w(x) + g'(\bar u(x);w(x)) = 2\sqrt{2\alpha\beta}|w|$. A similar argument leads to the same equality when $w(x) > 0$ and $\bar\varphi(x)= \sqrt{2\alpha\beta}$.
Let now $x\in \Omega$ such that $w(x) \ne 0$ and $|\bphi(x)| \le \sqrt{2\alpha\beta} -\tau$, implying $\bar u(x)=0$ by Lemma \ref{L4.2}. Then
we find $\bphi(x) w(x)  + g'(\bar u(x);w(x)) \ge (\tau - \sqrt{2\alpha\beta} + \sqrt{2\alpha\beta})|w| = \tau|w|$.
 Finally, if $|\bar\varphi(x)| > \alpha\gamma + \tau$ and $w(x) \neq 0$ we infer from Lemma \ref{L4.2} and the fact that $w \in T_{\uad}(\bar u)$ that
\[
\bphi(x) w(x)  + g'(\bar u(x);w(x)) = |\bphi(x)| |w(x)| - \alpha\gamma |w(x)| \ge \tau|w(x)|.
\]
Using \eqref{eq42}, we obtain
\[
 F'(\bar u)w + G'(\bar u;w)\ge \int_{\Omega_{\bar u,w}}\bphi(x) w(x)  + g'(\bar u(x);w(x))\dx
\ge \tau\|w\|_{L^1(\Omega_{\bar u,w})},
\]
which is the claim.
\end{proof}

Unfortunately, there are no remainder term estimates of $G$ of the type
\[
G(u+h) - G(u) - G'(u;h) - \frac12G''(u;h^2) \ge o(\|h\|_{L^2(\Omega_u)}^2)
\]
available, cf., Lemma \ref{L4.7} and Remark \ref{R4.8}. To overcome this difficulty, we will replace $G''$ by
\begin{equation}\label{eqgtilde}
 \tilde G(\bar u; v^2) :=
 \alpha
 \int_{\{ x\in \Omega:\ |\bar u(x)|\ge \sqrt{\frac{2\beta}\alpha}, \ \sign(v(x))=\sign(\bar u(x))\}} v^2 \dx.
\end{equation}
in the second-order condition. Clearly, $G''(\bar u;v^2)\ge \tilde G(\bar u; v^2)\ge0$ holds.
In addition, we have
\begin{lemma}\label{L4.14}
Let $u,v,h\in L^2(\Omega)$, then  the inequalities
 \[\begin{gathered}
  G(u+h) - G(u) - G'(u;h) - \frac12 \tilde G(u;h^2) \ge 0, \\
  G'(v;u-v) \le - G'(u;v-u) - \tilde G(u; (v-u)^2)  \le - G'(u;v-u)
  \end{gathered}
 \]
 are satisfied.
\end{lemma}
\begin{proof}
 The first claim is a consequence of Lemma \ref{L4.6}.  Let us prove the second claim.
 From the convexity of $g$ we get that
\begin{multline*}
g'(v(x);u(x) - v(x)) + g'(u(x);v(x) - u(x))\\
 \le [g(u(x)) - g(v(x))] + [g(v(x)) - g(u(x))] = 0
\end{multline*}
for a.a.~$x\in \Omega$. In addition, if $|u(x)|\ge \sqrt{\frac{2\beta}\alpha}$ and $\sign(v(x)- u(x))=\sign(u(x))$,
which implies $|v(x)|\ge |u(x)| \ge \sqrt{\frac{2\beta}\alpha}$, then from \eqref{E4.1} we get
\[
g'(v(x);u(x)-v(x)) + g'(u(x);v(x)-u(x))  =- \alpha (v(x)-u(x))^2.
\]
Integrating the above inequalities we infer
\[
 G'(v;u-v) + G'(u;v-u) \le - \tilde G(u; (v - u)^2),
\]
which is the second claim.
\end{proof}

\begin{theorem}\label{T4.15}
 Let $\bar u$ satisfy the necessary optimality conditions of \Pbpc.
 Assume there exists $\delta>0$ and $\tau>0$ such that
 \[
  F''(\bar u)v^2 + \tilde G(\bar u; v^2)  \ge \delta \|z_v\|_{L^2(\Omega)}^2\quad \forall v\in C^\tau_{\bar u}.
 \]
Then there is $\rho>0$ and $\kappa>0$ such that
\[
 F(\bar u) + G(\bar u) + \kappa \|z_{u-\bar u}\|_{L^2(\Omega)}^2\le F(u) + G(u)
\]
for all $u\in B_\rho(u) \cap \uad$.
\end{theorem}
\begin{proof}
Without loss of generality we can assume that $\tau < \sqrt{2\alpha\beta}$. We follow the proof of \cite[Theorem 3.1]{CasasMateos2019}.
 The positive number $\rho$ will be determined in the course of the proof.
 Take $u\in B_\rho(\bar u) \cap \uad$.
 Let us distinguish the following cases.

\paragraph{Case 1: $u-\bar u \not \in E^\tau_{\bar u}$}
Then we can expand and use the property of $E^\tau_{\bar u}$ to estimate
with some $\theta\in(0,1)$
 \begin{multline*}
  F(u) + G(u) - (F(\bar u) + G(\bar u))\\
  \ge F'(\bar u)(u-\bar u) + G'(\bar u;u-\bar u) +
  (F'(\bar u + \theta(u-\bar u)) -F'(\bar u))(u-\bar u)\\
  > \tau \|z_{u-\bar u}\|_{L^2(\Omega)} -\rho \|F'(\bar u + \theta(u-\bar u)) -F'(\bar u)\|_{L^2(\Omega)}.
 \end{multline*}
According to  Theorem \ref{T2.1}, there is $\rho'>0$  and $C>0$ such that
\[
\|F'(\bar u + \theta(u-\bar u)) -F'(\bar u)\|_{L^2(\Omega)} \le C \|z_{u-\bar u}\|_{L^2(\Omega)}
\]
if $u \in B_{\rho'}(\bar u)$.
Let $\rho_1: = \min(\rho', \ \frac{\tau}{2C})$.
Then for $\rho\in (0,\rho_1)$, we obtain
\[
  \rho \|F'(\bar u + \theta(u-\bar u)) -F'(\bar u)\|_{L^2(\Omega)} \le  \frac\tau2 \|z_{u-\bar u}\|_{L^2(\Omega)},
\]
which proves
\[
   F(u) + G(u) - (F(\bar u) + G(\bar u)) \ge \frac\tau2 \|z_{u-\bar u}\|_{L^2(\Omega)} .
\]
By Theorem \ref{T2.3}, there is $c>0$ such that
\begin{equation}\label{eq403}
 \|z_{u-\bar u}\|_{L^2(\Omega)}^2 \le c \rho  \|z_{u-\bar u}\|_{L^2(\Omega)},
 \end{equation}
which finishes the proof of this case.

\paragraph{Case 2: $u-\bar u  \in C^\tau_{\bar u}$}
Using Lemma \ref{L4.14}, we can expand with $\theta\in (0,1)$
 \begin{multline*}
   F(u) + G(u) - (F(\bar u) + G(\bar u))\\
  \ge \frac12 F''(\bar u)(u-\bar u)^2 + \frac12 \tilde G(\bar u; (u-\bar u)^2)
  +\frac12(F''(\bar u + \theta(u-\bar u)) -F''(\bar u))(u-\bar u)^2\\
  \ge \frac\delta2 \|z_{u-\bar u}\|_{L^2(\Omega)}^2
  +\frac12(F''(\bar u + \theta(u-\bar u)) -F''(\bar u))(u-\bar u)^2
\end{multline*}
By Theorem \ref{T2.2}, there is $\rho_2>0$ such that
\[
 |F''(\bar u + \theta(u-\bar u)) -F''(\bar u))(u-\bar u)^2| \le \frac\delta2 \|z_{u-\bar u}\|_{L^2(\Omega)}^2
\]
holds for all $u\in B_{\rho_2}(\bar u)\cap \uad$.
This implies
\[
  F(u) + G(u) - (F(\bar u) + G(\bar u)) \ge  \frac\delta4 \|z_{u-\bar u}\|_{L^2(\Omega)}^2.
\]

\paragraph{Case 3: $u-\bar u\in E^\tau_{\bar u} \setminus D^\tau_{\bar u}$}
Let  $\Omega_{\bar u,u-\bar u}$ be as in Lemma \ref{L4.13}, i.e., it is the set of points, where $u-\bar u$ violates the
pointwise conditions in the definition of $D^\tau_{\bar u}$. Then
we split $u-\bar u= v + w$ with $v:= (1-\chi_{\Omega_{\bar u,u-\bar u}}) (u-\bar u)\in D^\tau_{\bar u}$ and
$w:=  \chi_{\Omega_{\bar u,u-\bar u}} (u-\bar u)$.
Then Lemma \ref{L4.13} implies
\begin{equation}\label{eq404}
 F'(\bar u)w + G'(\bar u;w) \ge \tau\|w\|_{L^1(\Omega)}.
\end{equation}
In the next step, we show that there exists $\tau' \in (0,\tau]$ such that
if $u-\bar u \in E^{\tau'}_{\bar u}$, then  $v\in E^{\tau'}_{\bar u} \subset E^\tau_{\bar u}$.
Using Theorem \ref{T2.3}, we deduce with \eqref{eq404}
\[
  F'(\bar u)w + G'(\bar u;w) \ge c_\tau \|z_w\|_{L^2(\Omega)}
\]
for $c_\tau>0$. Take $\tau' :=\min(\tau,c_\tau)$.
Suppose $u-\bar u \in E^{\tau'}_{\bar u}$.
Since $v$ and $w$ have disjoint support, it holds
\[
 G'(\bar u; u-\bar u) = G'(\bar u; v+w) = G'(\bar u;v) + G'(\bar u;w).
\]
Then we obtain
\[\begin{aligned}
 F'(\bar u)v + G'(\bar u;v) & = F'(\bar u;u-\bar u) + G'(\bar u; u-\bar u)
 - F'(\bar u)w - G'(\bar u; w) \\
 & \le \tau' \|z_{u-\bar u}\|_{L^2(\Omega)} - c_\tau  \|z_w\|_{L^2(\Omega)}\\
 & \le \tau' \|z_v\|_{L^2(\Omega)} + (\tau'- c_\tau) \|z_{w}\|_{L^2(\Omega)} \le \tau' \|z_v\|_{L^2(\Omega)},
\end{aligned}
\]
and $v\in E^{\tau'}_{\bar u}$ follows.
We now study the two cases $u-\bar u \in E^{\tau'}_{\bar u}$ and $u-\bar u \not\in E^{\tau'}_{\bar u}$.
\paragraph{Case 3a: $u-\bar u\in E^\tau_{\bar u} \setminus D^\tau_{\bar u}$ and $u-\bar u \in E^{\tau'}_{\bar u}$}
As argued above, this implies $v\in E^{\tau'}_{\bar u}\subset E^\tau_{\bar u}$
and $v\in D^\tau_{\bar u} \cap E^\tau_{\bar u}$. Hence the second-order condition applies to $v$.

Using Lemma \ref{L4.14}, \eqref{eq41}, and \eqref{eq404} above, we find
\begin{multline*}
F(u) + G(u) - (F(\bar u) + G(\bar u)) \\
\begin{aligned}
  &\ge F'(\bar u)(v+w) + \frac12 F''(\bar u)(v+w)^2
  + \frac12 (F''(u_\theta)-F''(\bar u))(v+w)^2\\
 &\qquad + G'(\bar u;v+w)
  + \frac12 \tilde G(\bar u; v^2) \\
 & \ge c_\tau \|z_w\|_{L^2(\Omega)} + \frac\delta2 \|z_v\|_{L^2(\Omega)}^2 \\
 &\qquad + F''(\bar u)(v,w) +\frac12 F''(\bar u)w^2+  \frac12 (F''(u_\theta)-F''(\bar u))(v+w)^2.
\end{aligned}
\end{multline*}
Due to \eqref{E2.6}, there is $M>0$ such that
\[
 |F''(\bar u)(v,w)| \le M \|z_v\|_{L^2(\Omega)} \|z_w\|_{L^2(\Omega)}
 \le \frac\delta8 \|z_v\|_{L^2(\Omega)}^2 + \frac{2M^2}\delta \|z_w\|_{L^2(\Omega)}^2
\]
By Theorem \ref{T2.2}, there is $\rho_{3\textrm a}>0$ such that
\[
 |(F''(u_\theta)-F''(\bar u))(v+w)^2| \le \frac\delta4 (\|z_v\|_{L^2(\Omega)}^2+ \|z_w\|_{L^2(\Omega)}^2)
\]
for all $u\in B_{\rho_{3\textrm a}}(\bar u)\cap \uad$.
Collecting these estimates yields with some $K>0$
\[
F(u) + G(u) - (F(\bar u) + G(\bar u))
\ge (c_\tau - K\|z_w\|_{L^2(\Omega)} )  \|z_w\|_{L^2(\Omega)} + \frac{\delta}{8} \|z_v\|_{L^2(\Omega)}^2.
\]
Decreasing $\rho_{3\textrm a}$ if necessary, we can achieve
 $c_\tau - K\|z_w\|_{L^2(\Omega)}\ge \frac{\delta}8\|z_w\|_{L^2(\Omega)}$.
Using \eqref{eq403} and $\|z_w\|_{L^2(\Omega)}^2 + \|z_v\|_{L^2(\Omega)}^2 \ge \frac12\|z_{u-\bar u}\|_{L^2(\Omega)}^2$
concludes this case.

\paragraph{Case 3b: $u-\bar u\in E^\tau_{\bar u} \setminus D^\tau_{\bar u}$ and $u-\bar u \not\in E^{\tau'}_{\bar u}$}
This case was already studied (with different parameter) in Case 1, proving optimality in a ball $B_{\rho_{3\textrm b}}(\bar u)$.

Taking the $\rho:=\min(\rho_1,\rho_2,\rho_{3\textrm a},\rho_{3\textrm b})$ proves the claim.
\end{proof}

\begin{corollary}
 There is $\beta^*$ such that for all $\beta>\beta^*$ the control $\bar u=0$ is locally optimal.
\end{corollary}
\begin{proof}
Let us denote by $\varphi_0$ the adjoint state associated to $\bar u:=0$.
Take $\tau>0$, and set $\beta^*$ such that $\|\varphi_0\|_{L^\infty(\Omega)} \le \sqrt{2\alpha\beta^*} - \tau$.
Then for $\beta>\beta^*$, it holds $D^\tau_{\bar u}=\{0\}$, and the second-order condition is trivially fulfilled.
Therefore, $\bar u=0$ is a local solution of \Pbpc.
\end{proof}

\begin{remark}
 In the proof of Theorem \ref{T4.15}, we only used the following condition:
 \[
    F''(\bar u)(u-\bar u)^2 + \int_{\{x: |\bar u|\ge\sqrt{\frac{2\beta}\alpha}, |u|\ge \sqrt{\frac{2\beta}\alpha}\}}
    (u-\bar u)^2 \dx  \ge \delta \|z_{u-\bar u}\|_{L^2(\Omega)}^2\quad \forall u-\bar u\in C^\tau_{\bar u}
 \]
Of course, the expression on the left-hand side is not a bilinear form.
\end{remark}

\begin{remark}
Let us remark that a condition of the type
 \[
  F''(\bar u)v^2 + \tilde G(\bar u; v^2)  \ge \delta \|v\|_{L^2(\Omega)}^2
 \]
for test functions $v$ in some cone cannot be expected to holds, as
$v\mapsto\tilde G(\bar u; v^2)$ is not coercive on $L^2(\Omega)$.
Hence, we have to resort to the weaker condition, which is also used in \cite{Casas2012} for bang-bang control
problems.
\end{remark}

\begin{theorem}\label{T4.18}
 Let $\bar u$ satisfy the necessary optimality conditions of \Pbpc.
 Assume there exists $\delta>0$ and $\tau>0$ such that
 \[
  F''(\bar u)v^2 + \tilde G(\bar u; v^2)  \ge \delta \|z_v\|_{L^2(\Omega)}^2\quad \forall v\in C^\tau_{\bar u}.
 \]
Then $\bar u$ is an isolated stationary point of  \Pbpc.
\end{theorem}
\begin{proof}
We follow the proof of Theorem \ref{T4.15} above.
We will show that there is $\rho>0$ such that
$B_\rho(\bar u) \cap \uad$ does not contain a stationary point of \Pbpc different from $\bar u$.
The positive number $\rho$ will be determined in the course of the proof. Take $u\in B_\rho(\bar u) \cap \uad$, $u\ne\bar u$.
We will show that if $\rho$ is small enough then the inequality
$ 
 F'(u)(\bar u-u) +  G'(u;\bar u-u) <0
$ 
holds, and $u$ cannot be a stationary point.
Let us {note} that $u\ne\bar u$ implies $z_{u-\bar u}\ne0$.
Again, we will distinguish the following cases.

\paragraph{Case 1: $u-\bar u \not \in E^\tau_{\bar u}$}

From the convexity of $G$ we get that
\[
G'(u;\bar u-u) + G'(\bar u;u-\bar u) \le [G(\bar u) - G(u)] + [G(u) - G(\bar u)] = 0.
\]
We use this inequality, the property of $E^\tau_{\bar u}$,  and Theorem \ref{T2.1} to estimate
\begin{multline*}
F'(u)(\bar u-u) +  G'(u;\bar u-u) \le -( F'(\bar u)(u-\bar u) + G'(\bar u;u-\bar u) ) +  (F'(\bar u)-F'(u))(u-\bar u) \\
\le (C \|u-\bar u\|_{L^2(\Omega)} - \tau )\|z_{u-\bar u}\|_{L^2(\Omega)} .
\end{multline*}
Clearly, this expression is negative if $\rho<\rho_1:= \frac{\tau}{2C}$.

\paragraph{Case 2: $u-\bar u  \in C^\tau_{\bar u}$}
By Lemma \ref{L4.14}, we can expand with $u_\theta := \bar u + \theta (u-\bar u)$, $\theta\in (0,1)$,
\begin{multline*}
F'(u)(\bar u-u) +  G'(u;\bar u-u)\\
\begin{aligned}
& \le -( F'(\bar u)(u-\bar u) + G'(\bar u;u-\bar u) ) - \tilde G(\bar u;(u-\bar u)^2)
+  (F'(\bar u)-F'(u))(u-\bar u) \\
&\le -( F''(\bar u)(u-\bar u)^2 +  \tilde G(\bar u;(u-\bar u)^2)) + (F''(\bar u) - F''(u_\theta))(u-\bar u)^2\\
&\le -\delta \|z_{u-\bar u}\|_{L^2(\Omega)}^2  + (F''(\bar u) - F''(u_\theta))(u-\bar u)^2.
\end{aligned}
\end{multline*}
By Theorem \ref{T2.2}, there is $\rho_2>0$ such that $F'(u)(\bar u-u) +  G'(u;\bar u-u) <0$ holds if $\rho<\rho_2$.

Let us split $u-\bar u = v+w$ as in the proof of Theorem \ref{T4.15}.
 Let  $\tau'$ be as in that proof.
Then it remains to consider the following two cases.

\paragraph{Case 3a: $u-\bar u\in E^\tau_{\bar u} \setminus D^\tau_{\bar u}$ and $u-\bar u \in E^{\tau'}_{\bar u}$}

We obtain using Lemmas \ref{L4.13} and \ref{L4.14} with $u_\theta := \bar u + \theta (u-\bar u)$, $\theta\in(0,1)$,
\begin{multline*}
F'(u)(\bar u-u) +  G'(u;\bar u-u)\\
\begin{aligned}
 & \le -( F'(\bar u)(u-\bar u) + G'(\bar u;u-\bar u) ) - \tilde G(\bar u;(u-\bar u)^2)
+  (F'(\bar u)-F'(u))(u-\bar u) \\
&\le -c_\tau \|z_w\|_{L^2(\Omega)}  -( F''(\bar u)(u-\bar u)^2 +  \tilde G(\bar u;(u-\bar u)^2)) + (F''(\bar u) - F''(u_\theta))(u-\bar u)^2.
\end{aligned}
\end{multline*}
Arguing as in the proof  of Theorem \ref{T4.15}, we find
\[
F'(u)(\bar u-u) +  G'(u;\bar u-u) \le
- (c_\tau - K\|z_w\|_{L^2(\Omega)} )  \|z_w\|_{L^2(\Omega)} - \frac{\delta}{8} \|z_v\|_{L^2(\Omega)}^2.
\]
where the right-hand side is negative for $\rho < \rho_{3\textrm a}$,
since $z_{u-\bar u} = z_v+z_w \ne0$.

\paragraph{Case 3b: $u-\bar u\in E^\tau_{\bar u} \setminus D^\tau_{\bar u}$ and $u-\bar u \not\in E^{\tau'}_{\bar u}$}
This case is Case 1 with different parameters, proving the claim in a ball $B_{\rho_{3\textrm b}}(\bar u)$.

Taking the $\rho:=\min(\rho_1,\rho_2,\rho_{3\textrm a},\rho_{3\textrm b})$ proves the claim.
\end{proof}

\subsection{Second-order sufficient optimality condition for the original problem}

We will use the sufficient conditions for \Pbpc to obtain sufficient optimality conditions for \Pb. First, let us observe that stationary points of \Pb are stationary points of \Pbpc as well.

\begin{lemma}[{\cite[Lemma 3.25]{Wachsmuth2019}}]
\label{L4.19}
 Let $\bar u$ satisfy the PMP for \Pb. Then it is a stationary point of \Pbpc.
\end{lemma}

\begin{lemma}\label{L4.20}
 Let $\bar u$ satisfy the PMP for \Pb. Then
 \[
  g(\bar u(x)) = \frac\alpha2 |\bar u(x)|^2 + \beta|\bar u(x)|_0
 \]
for almost all $x\in \Omega$.
\end{lemma}
\begin{proof}
 Consider first the case $\alpha>0$. Then by Corollary \ref{C3.3}, we have $\bar u(x)=0$ or $|\bar u(x)|\ge\min(\sqrt{\frac{2\beta}\alpha}, \gamma)  = \sqrt{\frac{2\beta}\alpha}$.
 The claim follows using the implication \eqref{eq405}.
 In the case $\alpha=0$, we have $g(\bar u(x)) =\beta|\bar u(x)|_0$ if and only if $\bar u(x)\in \{-\gamma,0,\gamma\}$.
 The latter inclusion is valid for almost all $x$ due to Corollary \ref{C3.3}.
\end{proof}

\begin{theorem}\label{T4.21}
 Let $\bar u$ satisfy the PMP for \Pb.
 Assume there exist $\delta>0$ and $\tau>0$ such that
 \[
  F''(\bar u)v^2 + \tilde G(\bar u; v^2)  \ge \delta \|z_v\|_{L^2(\Omega)}^2\quad \forall v\in C^\tau_{\bar u}.
 \]
Then there are $\rho>0$ and $\kappa>0$ such that
\[
 J(\bar u) + \kappa \|z_{u-\bar u}\|_{L^2(\Omega)}^2\le J(u)
\]
for all $u\in B_\rho(u) \cap \uad$.
In addition, $(B_\rho(u) \cap \uad) \setminus\{\bar u\}$ does not contain a control satisfying the  PMP for \Pb.
\end{theorem}

The critical cones are defined in \eqref{eqdtau}--\eqref{eqctau}. The definition of $\tilde G$ is in \eqref{eqgtilde}.
\begin{proof}
 By Lemma \ref{L4.19}, $\bar u$ is stationary for \Pbpc. Theorem \ref{T4.15} implies that $\bar u$ is a local minimum of \Pbpc,
 and there are $\kappa>0$ and $\rho>0$ such that
\[
 F(\bar u) + G(\bar u) + \kappa \|z_{u-\bar u}\|_{L^2(\Omega)}^2\le F(u) + G(u)
\]
for all $u\in B_\rho(u) \cap \uad$. Take $u\in B_\rho(u) \cap \uad$.
Then we have the following
\[
  J(\bar u) + \kappa \|z_{u-\bar u}\|_{L^2(\Omega)}^2
  = F(\bar u) + G(\bar u)  + \kappa \|z_{u-\bar u}\|_{L^2(\Omega)}^2
  \le F(u) + G(u) \le J(u),
\]
where the first equality is due to Lemma \ref{L4.20}, and the last inequality follows from properties of the convex envelope.

By decreasing $\rho$ if necessary, Theorem \ref{T4.18} yields that $\bar u$ is an isolated stationary point for \Pbpc.
Since controls satisfying PMP for \Pb are stationary for \Pbpc by Lemma \ref{L4.19}, the claim follows.
\end{proof}

Let us briefly compare the bilinear forms that appear in the second-order necessary and sufficient optimality.
First, Theorem \ref{T3.6} states that
\[
  F''(\bar u)(v,v) + \alpha \|v\|_{L^2(\Omega)}^2 \ge 0
\]
for all $v \in C_{\bar u}$, where the critical cone $C_{\bar u}$ is given by
\[
 C_{\bar u}  = \{ v\in T_{\uad}(\bar u): \ v(x)=0 \text{ if } \bar u(x)=0 \text{ or } \bar \varphi(x) + \alpha\bar u(x)\ne0\}.
\]
The sufficient condition in Theorem \ref{T4.21} is based on the following inequality
\[
F''(\bar u)v^2 + \tilde G(\bar u; v^2)  \ge \delta \|z_v\|_{L^2(\Omega)}^2\quad \forall v\in C^\tau_{\bar u}.
\]
Clearly, it holds
$ 
 \alpha \|v\|_{L^2(\Omega)}^2  \ge \tilde G(\bar u; v^2)
$ 
for $v\in  C_{\bar u}$, and the inequality is strict
if the conditions $\sqrt{\frac{2\beta}\alpha}\le |\bar u(x)|<\gamma$ and $\sign(v(x))\ne\sign(\bar u(x))\}$
are satisfied on a set of positive measure.

\bigskip
Let us comment on the possibility to use some structural assumptions on $\bar u$ and $\varphi$ in order
to be able to weaken the sufficient optimality conditions. Such assumptions were used in \cite{CasasWachsmuthWachsmuth2017,CasasWachsmuthWachsmuth2018} for bang-bang control
problems.

In our case, one possibility would be to assume the following:
Let $\bar u$ with adjoint state $\bar\varphi$ be a stationary point of \Pbpc. Assume there is $c>0$ such that
\[
 \left| \left\{ x\in\Omega: \  \sqrt{2\alpha\beta}-\epsilon< |\bar\varphi(x)| \le \sqrt{2\alpha\beta} \right\}\right| \le c\epsilon
\]
for all $\epsilon \in (0,\sqrt{2\alpha\beta})$.
In addition, let us assume $\sqrt{\frac{2\beta}\alpha} < \gamma< +\infty$.

This assumption implies that the measure of the set $\{x\in\Omega: \ |\bar\varphi(x)|=\sqrt{2\alpha\beta}\}$ is zero.
As a consequence, by Lemma \ref{L4.2} we have that the measure of the set $\{x\in\Omega: \ 0<|\bar u(x)|<\sqrt{\frac{2\beta}\alpha}\}$ is zero.
Using this assumption, we get some additional growth from first-order expressions, compare Lemma \ref{L4.13}.

\begin{lemma}
Let $\bar u$ with adjoint state $\bar\varphi$ be a stationary point of \Pbpc. Assume that the structural assumption is fulfilled. Then
there is $\kappa>0$ such that
\[
 F'(\bar u)(u-\bar u) + G'(\bar u; u-\bar u) \ge\kappa \|u-\bar u\|_{L^1(\Omega_0)}^2
\]
is satisfied for all $u\in U_{ad}$, where $\Omega_0 := \{x\in\Omega: \ \bar u(x)=0\}$.
\end{lemma}
\begin{proof}
 Let $u\in U_{ad}$ be given.
 For  $\epsilon \in (0,\sqrt{2\alpha\beta})$ let us introduce the the set
 \[
  \Omega_{0,\epsilon} := \{x\in\Omega: \ \sqrt{2\alpha\beta}-\epsilon< |\bar\varphi(x)| < \sqrt{2\alpha\beta}\}.
 \]
 By definition, we have $\Omega_{0,\epsilon} \subset \Omega_0$.
 In addition, the measure of  $\Omega_{0,\epsilon} $ is bounded by $c\epsilon$ according to the structural assumption.
 Since $\bar u$ is stationary, we have
 \[
   F'(\bar u)(u-\bar u) + G'(\bar u; u-\bar u) \ge \int_{\Omega_0} \bar\varphi (u-\bar u) + g'(\bar u;u-\bar u) \dx.
 \]
 On $\Omega_0$, we have $|\bar\varphi(x)| \le\sqrt{2\alpha\beta}$ and $g'(\bar u;u-\bar u) =\sqrt{2\alpha\beta}|u-\bar u|$.
 Hence, it holds
\begin{multline*}
  \int_{\Omega_0} \bar\varphi (u-\bar u) + g'(\bar u;u-\bar u) \dx =  \int_{\Omega_0}\bar\varphi (u-\bar u) +\sqrt{2\alpha\beta}|u-\bar u|\dx\\
  \begin{aligned}
  &\ge \int_{\Omega_0\setminus \Omega_{0,\epsilon}}\bar\varphi (u-\bar u) +\sqrt{2\alpha\beta}|u-\bar u| \dx\\
  & \ge \epsilon\|u-\bar u\|_{L^1(\Omega_0\setminus \Omega_{0,\epsilon})}
   = \epsilon\|u-\bar u\|_{L^1(\Omega_0)} -\epsilon \|u-\bar u\|_{L^1(\Omega_{0,\epsilon})}\\
  & \ge\epsilon\|u-\bar u\|_{L^1(\Omega_0)} -c \epsilon^2 \gamma.
\end{aligned}
\end{multline*}
Setting $\epsilon:= \|u-\bar u\|_{L^1(\Omega_0)} (2c\gamma)^{-1}$ yields the claim
with $\kappa:=(4c\gamma)^{-1}$.
\end{proof}

This result could be used to reduce the cone $ D^\tau_{\bar u}$, see \eqref{eqdtau}, to
\[
D^\tau_{\bar u}{}':=\{ v\in T_\uad(\bar u): \
   v(x) = 0 \text{ if } |\bphi(x)| <\sqrt{2\alpha\beta} \text{ or }|\bphi(x)| \ge \alpha\gamma + \tau \},
\]
However, in light of Lemma \ref{L4.6} and Remark \ref{R4.8}, it seems impossible to devise a structural assumption to obtain
\[
  F'(\bar u)(u-\bar u) + G'(\bar u; u-\bar u) \ge\kappa \|u-\bar u\|_{L^1(\Omega)}^2\quad \forall u\in U_{ad},
\]
as it was done in \cite{CasasWachsmuthWachsmuth2017,CasasWachsmuthWachsmuth2018} for bang-bang control problems.

\bibliographystyle{siam}
\bibliography{CW2019}
\end{document}